\documentclass{svproc}

\usepackage{amscd,amsfonts,mathrsfs,amssymb}
\usepackage{comment}
\usepackage{mathtools}
\usepackage{enumerate}
\usepackage{multicol}        			
\usepackage{color}           			
\usepackage{array}
\usepackage{ae}
\usepackage{epsfig}
\usepackage[e]{esvect}

\usepackage{empheq}
\usepackage[latin1]{inputenc}

\usepackage{hyperref}
\usepackage{amsrefs}

\usepackage{calrsfs}
\DeclareMathAlphabet{\pazocal}{OMS}{zplm}{m}{n}

\newcommand{\emp}{}

\newcommand{\R}{\mathbf{R}}

 \newcommand{\RR}{\mathbf{R}}  

\newcommand{\ee}{\mathbf e}

\def\dt       {\frac{d}{dt}\,}
\def\partialt       {\frac{\partial}{\partial t}\,}

\def\pproof#1{\@ifnextchar[\opargproof
{\opargproof[\it Proof of #1.]}}
\def\opargproof[#1]{\par\noindent {\bf #1 }}

\makeatother
\newcommand{\grim}{\mathcal{G}}

\newcommand{\DELTA}{\operatorname{\Delta}}

\numberwithin{equation}{section}

\begin{document}
\mainmatter  
\title{Notes on translating solitons for Mean Curvature Flow}
\titlerunning{Translating solitons for MCF}
\author{David Hoffman\inst{1} \and Tom Ilmanen\inst{2} \and Francisco Martín\inst{3}\thanks{F. Martín was partially supported by the MINECO/FEDER grant MTM2017-89677-P and  by the
Leverhulme Trust grant IN-2016-019.} \and Brian White\inst{4}\thanks{B. White was partially supported by grants from the Simons Foundation (\#396369) and from the National Science Foundation (DMS~1404282, DMS~1711293).}}
\authorrunning{D. Hoffman et al.}

\institute{Department of Mathematics,
 Stanford University, 
   Stanford, CA 94305, USA\\
\email{dhoffman@stanford.edu}
\and
Department of Mathematics,
E. T. H. Zürich, 
Rämistrasse 101, 
8092 Zürich,
Switzerland\\
\email{tom.ilmanen@math.ethz.ch}
\and
Departamento de Geometría y Topología, 
Universidad de Granada, 18071 Granada, Spain\\
\email{fmartin@ugr.es}
\and
Department of Mathematics,
 Stanford University,
   Stanford, CA 94305, USA\\
\email{bcwhite@stanford.edu}
}

\maketitle

\begin{abstract}
These notes provide an introduction to 
 translating solitons for the mean curvature flow in $\R^3$.
 In particular, we describe  a full classification of the translators that 
are complete graphs over domains in $\R^2$. 
\keywords{Mean curvature flow, singularities, monotonicity formula, area estimates, comparison principle.}
\end{abstract}

\section{Introduction}

\label{sec:intro}
Mean curvature flow is an exciting area of mathematical research. It is situated at the crossroads of several scientific disciplines: geometric analysis, geometric measure theory, partial differential equations, differential topology, mathematical physics, image processing, computer-aided design, among others. 
In these notes, we give a brief introduction to mean curvature flow
and we describe recent progress on translators, an important class of solutions
to mean curvature flow.

In physics, {\emp diffusion} is  a process which equilibrates spatial variations in concentration. If we consider a
initial concentration $u_0$ on a domain $\Omega \subseteq \RR^{2}$ and seek solutions of the linear {\emp heat equation}
$\frac{\partial}{\partial t} u -\DELTA u=0, $
with initial data $u_0$ and natural boundary conditions on $\partial \Omega$, we obtain  {\emp smoothed}
concentrations $\{u_t\}_{t>0}$. 
\begin{figure}[h]
\includegraphics[width=.9 \textwidth]{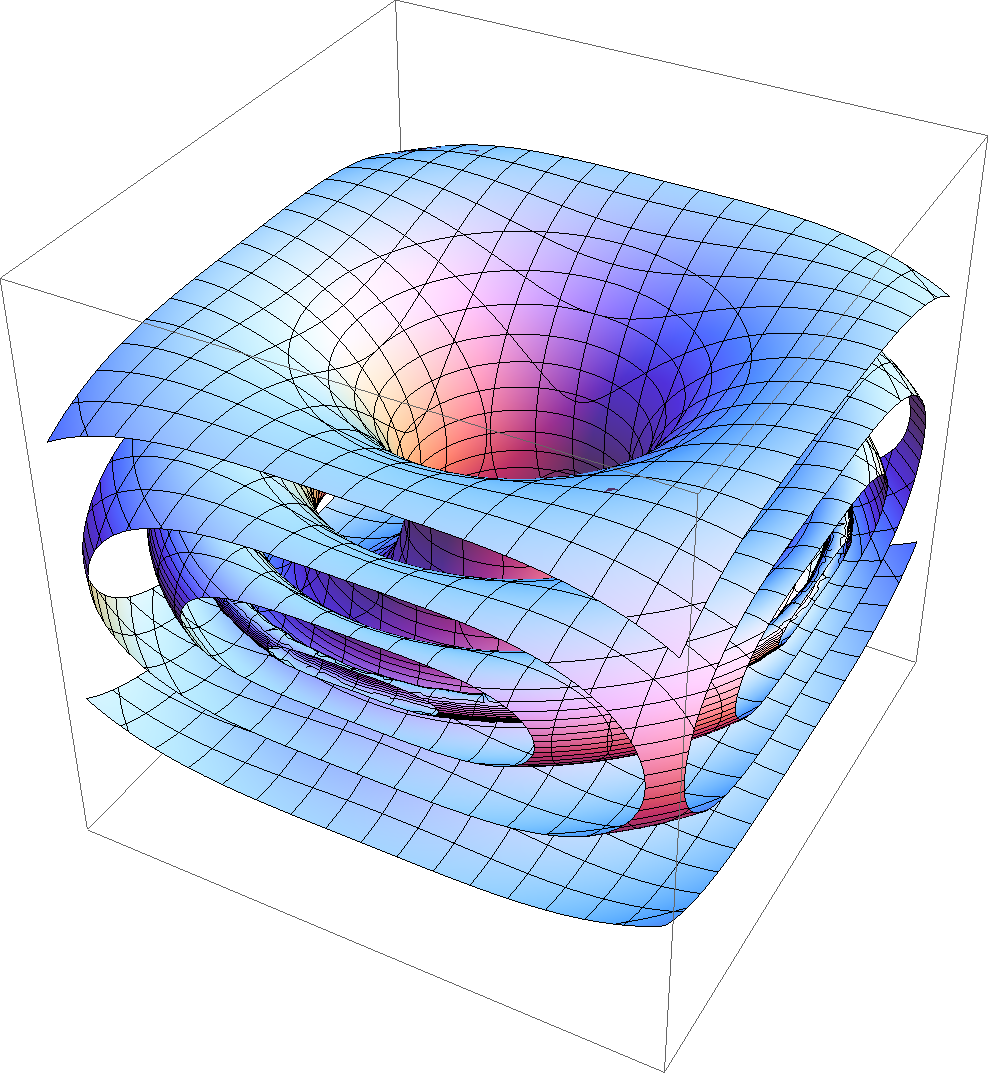}\caption{A surface moving by mean curvature.}\label{conejo}
\end{figure}
If we are interested in the {\emp smoothing of perturbed surface geometries}, it make sense to use analogous 
strategies.
The geometrical counterpart of the Euclidean Laplace operator $\DELTA$ on a smooth surface $M^2 \subset \RR^{3}$ (or more generally, a hypersurface $M^n \subset \RR^{n+1}$) is the Laplace-Beltrami operator, which we will denote as $\DELTA_M$. Thus, we obtain the {\emp geometric diffusion equation} 
\begin{equation} \label{eq:intro2}
\frac{\partial}{\partial t}\; \mathbf{x}=\DELTA_{M_t}  \mathbf{x},
\end{equation}
for the coordinates {\bf x} of the corresponding family of surfaces $\{M_t\}_{t \in [0,T)}.$


A classical formula (see \cite{hildebrandt}, for instance) says that, given a hypersurface in Euclidean space, one has:
$$ \DELTA_{M_t} \mathbf{x} = \vec{H},$$
where $\vec{H}$ represents the {\emp mean curvature} vector. This means that \eqref{eq:intro2} can be written as:
\begin{equation} \label{eq:intro3}
\frac{\partial}{\partial t}\; \mathbf{x}(p,t)=\vec{H}(p,t).
\end{equation}
Using techniques of parabolic PDE's it is possible to deduce  the existence and uniqueness of the mean curvature flow for a small time period in the case of compact manifolds (for details see \cite{regularityTheoryMCF}, \cite{LecturesNotesOnMCF}, among others). 
\begin{theorem} \label{teo:existencia_y_unicidad_del_MCF}
Given a compact, immersed hypersurface $M$ in $\mathbf{R}^{n+1}$ then there exists a unique mean curvature flow defined on an interval $[0,T)$
with initial surface $M$.
\end{theorem}

The mean curvature is known to be the first variation of the area functional $M \mapsto \int_M d \mu$ (see \cite{colding-minicozzi, meeks-perez}.) We will obtain for the $\operatorname{Area} (\Omega(t))$ of a relatively compact $\Omega(t) \subset M_t$ that 
$$ \dt \left( \operatorname{Area}(\Omega(t) \right) = - \int_{\Omega(t)} |\vec{H}|^2 d\mu_t.$$
In other words, we get that the mean curvature flow is the corresponding {\emp gradient flow for the area functional}: 
\begin{remark}
\emp The mean curvature flow is the flow of steepest descent of surface area.
\end{remark}
Moreover, we also have a nice {\emp maximum principle} for this particular diffusion equation.
\begin{theorem}[Maximum/comparison principle] \label{th:compa}
If two compact immersed hypersurfaces of $\RR^{n+1}$ are initially disjoint, they remain so. Furthermore, compact embedded hypersurfaces remain embedded.
\end{theorem}
Similarly, convexity and mean convexity are preserved: 
\begin{itemize}
\item If the initial hypersurface $M$ is convex (i.e., all the geodesic curvatures are positive, or equivalently $M$ bounds a convex region of $\RR^{n+1}$), then $M_t$ is convex, for any $t$.
\item If $M$ is mean convex ($H>0$), then $M_t$ is also mean convex, for any $t$.
\end{itemize}
\par
Moreover, mean curvature flow has a property which is similar to the eventual simplicity for the solutions of the heat equation.
This result was proved by Huisken and asserts:
\begin{theorem}[\cite{huisken90}]
Convex, embedded, compact hypersurfaces converge to points $p \in \RR^{n+1}.$ After rescaling to keep the area
constant, they converge smoothly to round spheres.
\end{theorem}
As a consequence of the above theorems we have.

\begin{corollary}[\bf Existence of singularities in finite time] \label{co:nonsingular}
Let $M$ be a compact hypersurface in $\R^{n+1}$. If $M_t$ represents its evolution by the mean curvature flow, then $M_t$ must develop singularities in finite time. Moreover, if we denote this maximal time as $T_{\max}$, then
$$2\, n \, T_{\max} \leq  \left( \mbox{\rm diam}_{\R^{n+1}}(M) \right)^2.$$
\end{corollary}
\begin{proof}
Since $M$ is compact, it is contained an open ball
$B(p,\rho)$. So, $M$ must develop a singularity before the flow of $\mathbf{S}^n_{p}$ collapses at the point $p$, as otherwise we would  contradict the previous theorem. The upper bound of $T_{\max}$ is just a consequence of the collapsing time of a sphere.
\end{proof}

A natural question is: What can we say when $M$ is not compact? 
In this case, we can have long-time existence. 
A trivial example is the case of a complete, properly embedded minimal hypersurface $M$  in $\R^{n+1}$. Under the mean curvature flow, $M$ remains stationary, so the flow exists for all time. If we are looking for non-stationary examples, then we can consider the following example:
\begin{example}[\bf grim reapers]
Consider the euclidean
product $M=\Gamma\times\RR^{n-1}$, where $\Gamma$ is the {\it grim reaper} in $\RR^{2}$ represented by the
immersion \begin{equation} \label{grim-1}
f:(-\pi/2,\pi/2)\to\RR^{2} \end{equation} $$f(x)=(x,\log\cos x).$$
  If, ignoring parametrization, we let $M_t$ be the result of flowing $M$ by mean curvature flow for time $t$,
   then $M_t =  M - t \, e_{n+1}$, where again $\{\operatorname{e}_1, \ldots,\operatorname{e}_{n+1}\}$ represents the canonical basis of $\RR^{n+1}$. In other words, $M$ moves by vertical translations. By definition, we say that
$M$ is {\bf  a translating soliton} in the direction of $-\,\mbox{\rm e}_{n+1}$. More generally, any translator
in the direction of $-\,\mbox{e}_{n+1}$ which is a Riemannian product of a planar curve and an euclidean space $\RR^{n-1}$
can be obtained from this example by a suitable combination of a rotation and a dilation (see \cite{himw} for further details.)  We refer to these translating hypersurfaces as
 {\bf $n$-dimesional grim reapers}, or  simply grim reapers if the $n$ is clear from  the context.
\begin{figure}[h]
\begin{center}\includegraphics[width=.35\textwidth]{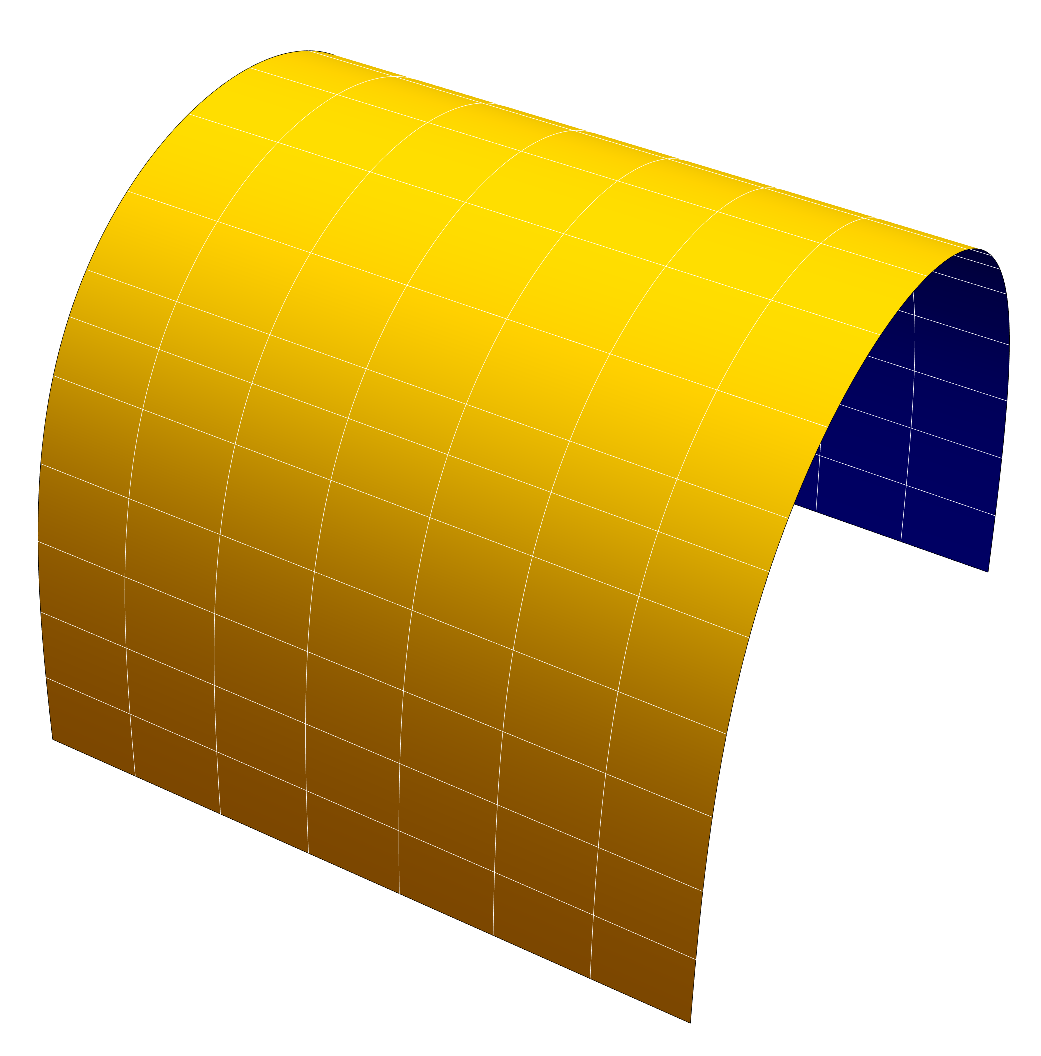} \end{center}\caption{A grim reaper}\label{grimreaperplane}
\end{figure}

\end{example}

\section{Some Remarks About Singularities}
\label{cap4:sec:singularidades}
Throughout this section, we consider a fixed compact initial hypersurface $M$. Consider the maximal time $T=T_M$ such that a smooth solution of the MCF $F: M \times [0,T) \to \RR^{n+1}$ as in Theorem \ref{teo:existencia_y_unicidad_del_MCF} exists.  
Then the embedding vector $F$ is uniformly bounded according to Theorem \ref{th:compa}. It follows that some spatial derivatives of the embedding $F_t$ have to become unbounded as $t \nearrow T$. Otherwise, we could apply Arzel\`a-Ascoli Theorem and obtain a smooth limit hypersurface, $M_T$, such that $M_t$ converges smoothly to $M_T$ as $t \nearrow T$. This is impossible because, in such a case, we could apply Theorem \ref{teo:existencia_y_unicidad_del_MCF} to restart the flow. In this way, we could extend the flow smoothly all the way up to $T+\varepsilon$, for some $\varepsilon>0$ small enough, contradicting the maximality of $T$. In fact, Huisken \cite[Theorem 8.1]{huisken84} showed that the second spatial derivative  (i.e, the norm
of the second fundamental form) blows up as $t \to T.$

We would like to say more about the ``blowing-up'' of the norm of $A$, as $t \nearrow T.$
The evolution equation for $|A|^{2}$ is 
$$\frac{\partial}{\partial t}|A|^{2}=\DELTA |A|^{2} -2\big|\nabla A \big|^{2}+2|A|^{4}.$$
Define $$|A|^2_{\text{max}}(t)\coloneqq \max_{M_{t}} |A|^2(\cdot, t).$$
Using Hamilton's trick (see \cite{LecturesNotesOnMCF}) we deduce that $|A|^2_{\max}$ is locally Lipschitz and that 
$$\dt |A|^2_{\max} (t_0)= \partialt |A|^2(p_0,t_0),$$
where $p_0$ is any point where $|A|^2(\cdot,t_0)$ reaches its maximum. 
Thus, using the above expression, we have
\begin{align*}
\dt |A|^2_{\max} (t_0) &= \partialt |A|^2(p_0,t_0)  \\
&= \DELTA |A|^{2}(p_0,t_0) -2\big|\nabla A(p_0,t_0)\big|^{2}+2|A|^{4}(p_0,t_0).
\end{align*}
It is well known that the Hessian of $|A|$ is negative semi-definite at any maximum. In particular the Laplacian of $|A|$ at these points is non-positive.
Hence,
$$\dt |A|^2_{\max} (t_0)\leq 2|A|^{4}(p_0,t_0) \leq  2|A|_{\text{max}}^{4}(t_0).
$$
Notice that $|A|_{\text{max}}^{2}$ is always positive, since otherwise  at some instant $t$ we would have $|A(\cdot,t)|\equiv 0$ along $M_{t}$, which would imply that  $M_{t}$ is totally geodesic  and therefore a hyperplane in $\mathbf{R}^{n+1}$, contrary to the fact that the initial surface was compact.

So,  one can prove that $1/|A|_{\text{max}}^{2}$ is locally Lipschitz. Then   the previous inequality allows us to deduce that:
$$-\frac{d}{dt}\left(\frac{1}{|A|_{\text{max}}^{2}} \right) \leq 2 \quad \mbox{\rm a.e. in $t\in [0,T)$.}$$
Integrating (respect to time)  in any  sub-interval $[t,s]\subset [0,T)$ we get
$$\frac{1}{|A|_{\text{max}}^{2}(t)}-\frac{1}{|A|_{\text{max}}^{2}(s)}\leq 2(s-t).$$

As  $|A(\cdot,t)|$ is not bounded as to tends to $T$, there exists a time sequence $s_{i} \nearrow T$ such that $$|A|_{\text{max}}^{2} (s_i)\to +\infty.$$ Substituting $s=s_{i}$ in the above inequality   and taking the limit as $i\to \infty$, we get
$$\frac{1}{|A|_{\text{max}}^{2}(t)}\leq 2(T-t).$$
We collect all this information in the next proposition.
\begin{proposition}
Consider the mean curvature flow for a compact initial hypersurface $M$. If $T$ is the maximal time of existence, then the following lower bound holds
$$|A|_{\text{max}}(t)\geq \frac{1}{\sqrt{2(T-t)}}$$
for all $t\in [0,T)$.

In particular,
$$\lim_{t\to T} |A|_{\text{max}}(t)=+\infty.$$
\end{proposition}
\begin{definition}
When this happens we say that $T$ is \textit{singular time } for the mean curvature flow.
\end{definition}

So we have the following improved version of Theorem \ref{teo:existencia_y_unicidad_del_MCF}:
\begin{theorem}
Given a compact, immersed hypersurface $M$ in $\mathbf{R}^{n+1}$ then there exists a unique mean curvature flow defined on a maximal interval $[0,T_{\text{max}})$.\newline
Moreover, $T_{\text{max}}$ is finite and
$$|A|_{\text{max}}(t)\geq \frac{1}{\sqrt{2(T_{\text{max}}-t)}}$$
for each $t\in [0,T_{\text{max}})$.\newline
\end{theorem}
\begin{remark}
From the above proposition, we deduce the following estimate for the maximal time of existence of flow:
$$T_{\text{max}}\geq \frac{1}{2|A|_{\text{max}}^{2}(0)}.$$
\end{remark}
\begin{definition}
Let $T$ be the maximal time of existence of the mean curvature flow. If there is a constant $C>1$ such that 
$$|A|_{\text{max}}(t)\leq \frac{C}{\sqrt{2(T-t)}},$$
then we say that the flow develops a  \textit{Type I singularity} at instant $T$.    
Otherwise, that is, if
$$\limsup_{t\to T} |A|_{\text{max}}(t) \, \sqrt{(T-t)}=+\infty,$$
we say that is a \textit{Type II singularity}.
\end{definition}
We conclude this brief section by pointing out  that there have been substantial breakthroughs in the study and understanding of the singularities of Type I, whereas Type II singularities have been much more difficult to study. This seems reasonable since, according to the above definition and the results we have seen, the singularities of Type I are those for which one has the best possible control of blow-up of the second fundamental form.

\section{Translators} \label{sec:translators}

A standard example of Type II singularity is given by a loop pinching off to a cusp (see Figure \ref{cardiod}). S. Angenent \cite{An1} proved, in the case of convex planar curves, that singularities of this kind are asymptotic (after rescaling) to the above mentioned grim reaper curve \eqref{grim-1}, which moves set-wise by translation. In this case, up to inner diffeomorphisms of the curve, it can be seen as a solution of the curve shortening flow which evolves by translations and is defined for all time.
In this paper we are interested in this type of solitons, which we will call {\bf translating solitons (or translators)} from now on. Summarizing this information,  we make the following definition:
\begin{definition}[Translator]
A {translator} is a hypersurface $M$ in $\R^{n+1}$ such that 
\[
   t\mapsto M- t \,\ee_{n+1}
\]
is a mean curvature flow, i.e., such that normal component of the velocity at each point is
equal to the mean curvature at that point:
\begin{equation}\label{general-translator-equation}
   \vec{H} = -\ee_{n+1}^\perp.
\end{equation}
\end{definition}

\begin{figure}[htbp]
\begin{center}
\includegraphics[width=.78\textwidth]{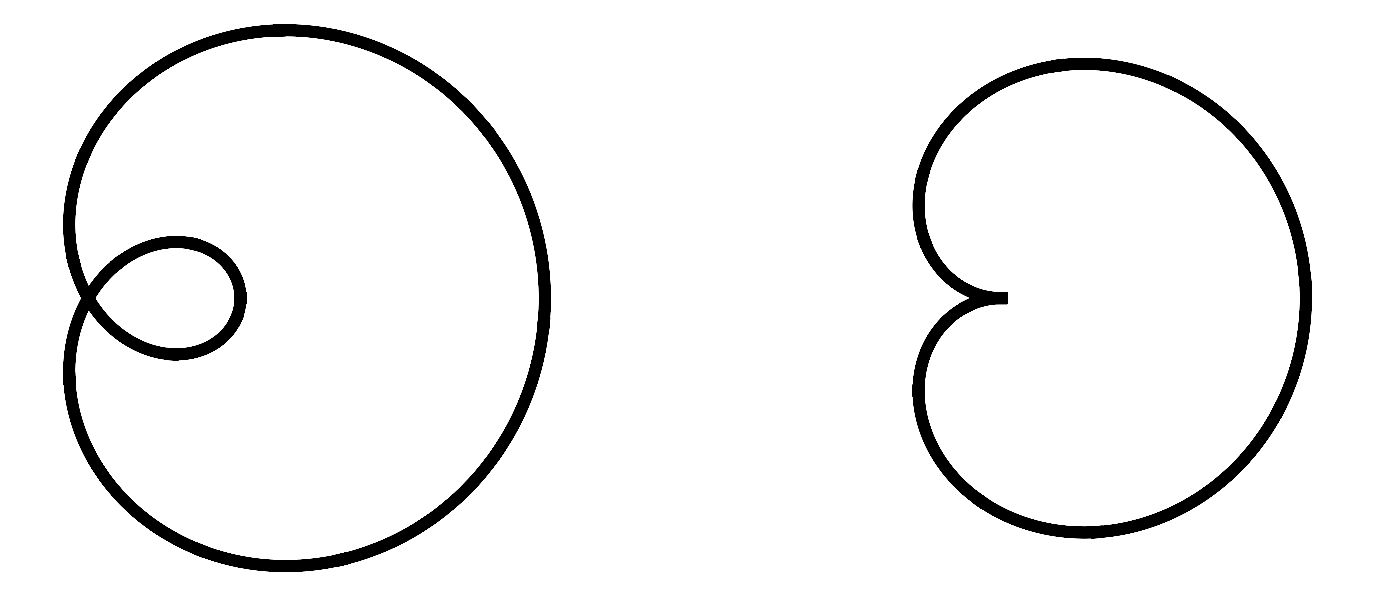}
\caption{}
\label{cardiod}
\end{center}
\end{figure}
The cylinder over a grim-reaper curve, i.e. the hypersurface in $\mathbf{R}^{n+1}$ parametrized by $\grim:\left(-\tfrac{\pi}{2},\tfrac{\pi}{2}\right)\times\mathbf{R}^{n-1}\longrightarrow \mathbf{R}^{n+1}$ given by $$\grim(x_{1},\ldots,x_{n})=(x_{1},\ldots,x_{n},-\log\cos x_{1}),$$ 
is a translating soliton.  It appears as limit of sequences of parabolic rescaled solutions of mean curvature flows of immersed mean convex hypersurfaces.  For example, we can take product of the loop pinching off to a cusp times $\R^{n-1}$.  We can produce others examples of solitons just by scaling and rotating the grim reaper. In this way, we obtain a  $1-$parameter family of translating solitons parametrized by $\grim_\theta:\left(-\tfrac{\pi}{2\cos(\theta)},\tfrac{\pi}{2\cos(\theta)}\right)\times\mathbf{R}^{n-1}\longrightarrow \mathbf{R}^{n+1}$
\begin{equation}\label{tiltedgrim}
\grim_\theta(x_{1},\ldots,x_{n})=(x_{1},\ldots,x_{n},\sec^2(\theta)\log\cos(x_{1}\cos(\theta))-\tan(\theta)x_{n}),
\end{equation}
where $\theta\in[0,\pi/2).$ Notice that the limit of the family $F_\theta$, as $\theta$ tends to $\pi/2$, is a hyperplane parallel to $\ee_{n+1}$.
\begin{figure}[htbp]
\begin{center}
\includegraphics[width=.65\textwidth]{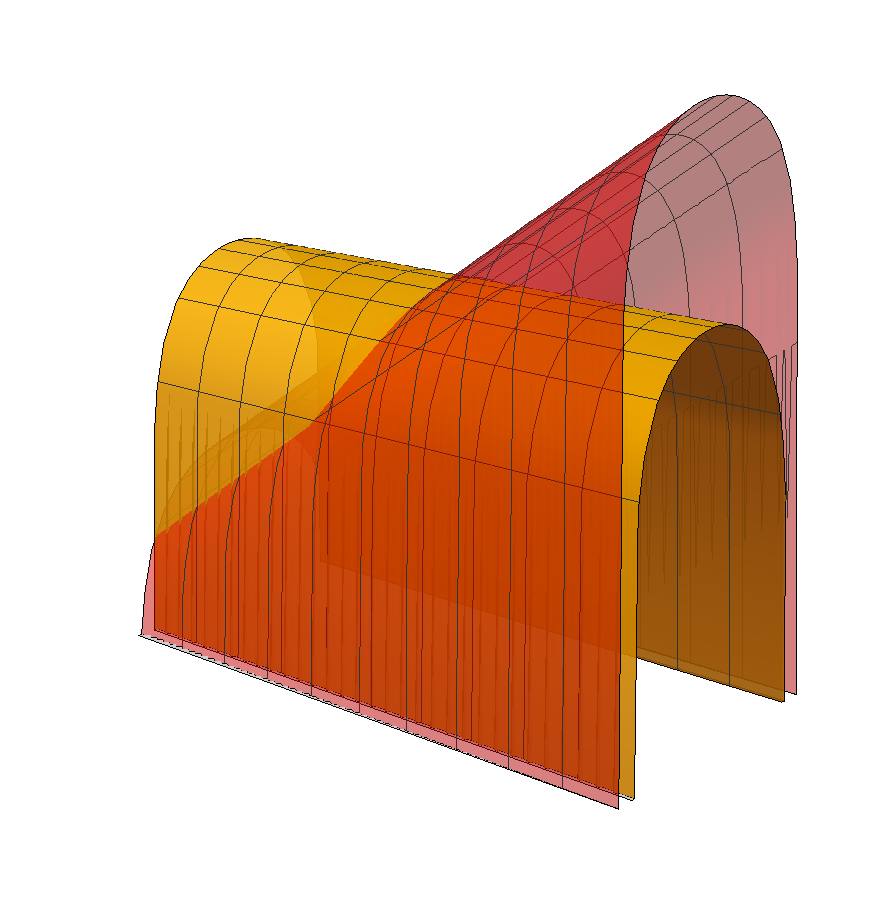}
\caption{The regular grim reaper in $\R^3$ and the tilted grim reaper for $\theta=\pi/6$.}
\label{default}
\end{center}
\end{figure}

\subsection{Variational approach}
Ilmanen~\cite{ilmanen} noticed that a translating soliton $M$ in $\RR^{n+1}$ can be seen as a minimal surface for the weighted volume functional 
\[
\mathcal{A}_f [M]= \int_M e^{-f}\, d\mu
\]
where $f$ represents the Euclidean height function, that is, the restriction of the last coordinate $x_{n+1}$ to $M$. We have the following 
\begin{proposition}[Ilmanen]
\label{first-variation}
Let $M^n$ be a translating soliton in $\R^{n+1}$ and let
{$N$ be its unit normal}.
 Then the {translator} equation 
\begin{equation}
\label{sol-scalar-bis}
H = \langle \ee_{n+1}, N\rangle
\end{equation}
on the relatively compact domain $\Omega \subset M$ is the Euler-Lagrange equation of the functional 
\begin{equation}
\label{vol-Mw}
\mathcal{A}_{f} [\Omega]={\rm vol}_{f} (\Omega) = \int_\Omega e^{-f}\, d\mu.
\end{equation} Moreover, the second variation formula for normal variations is given by
\begin{equation}
\label{second-variation-vol}
\delta^2\mathcal{A}_{f}[\Omega]\cdot (u,u)=-\int_M e^{-f}\,
u \, L_f  u \, d\mu,\quad u\in C^\infty_0(\Omega),
\end{equation}
where the stability operator $L_f$ is defined by
\begin{equation}
L_f u = \DELTA^{f} u + |A|^2 \, u
\end{equation}
where $\DELTA^f$ is the {\it drift Laplacian} given by
\begin{equation}
\label{deltaT-op}
\DELTA^{f} =\DELTA -\langle \nabla f, \nabla\, \cdot\,\rangle= \DELTA - \langle \ee_{n+1}, \nabla\,\cdot\,\rangle.
\end{equation}
\end{proposition}

\noindent \emph{Proof.} Given $\varepsilon >0$, let  $\Psi:(-\varepsilon, \varepsilon)\times M \to \RR^{n+1}$ be a variation of $M$ compactly supported in $\Omega\subset M$  with $\Psi(0,\,\cdot\,) = \mbox{Id}$ and
normal variational vector field
\[
\frac{\partial\Psi}{\partial s}\Big|_{s=0} = uN+T
\]
for some function $u\in C^\infty_0 (\Omega)$ and a tangent vector field $T\in \Gamma(TM)$. Here, $N$ denotes a local unit normal vector field along $M$.  Then
\begin{eqnarray*}
& &\frac{{\rm d}}{{\rm d}s}\Big|_{s=0}{\rm vol}_{f}[\Psi_s(\Omega)] =
\int_\Omega e^{-f}\,(\langle \ee_{n+1}, N\rangle-H) u \, d\mu + \int_\Omega {\rm div} (e^{-f}T)\, d\mu.
\end{eqnarray*}
Hence, stationary immersions for variations fixing  the boundary of $\Omega$ are characterized by the
scalar soliton equation
\[
H-\,\langle \ee_{n+1}, N\rangle =0\,\,\, \mbox{ on } \,\,\, \Omega \subset \subset M, 
\]
which yields (\ref{sol-scalar-bis}). Now we compute the second variation formula. At a stationary immersion
we have
\begin{eqnarray*}
& &\frac{{\rm d}^2}{{\rm d}s^2}\Big|_{s=0}{\rm vol}_{f}[\Psi_s(\Omega)] = 
\int_M e^{-f}\,\frac{d}{ds}\Big|_{s=0}(\langle \ee_{n+1},
N_s\rangle-H_s) u\, d\mu.
\end{eqnarray*}
Using the fact that
\begin{equation}
\label{der-N}
\bar\nabla_{\partial_s} N = -\nabla u -  \mathcal{W} T,
\end{equation}
(where $\mathcal{W}$ means the Weingarten map) then we compute
\begin{equation}
\label{der-angle}
\frac{{\rm d}}{{\rm d}s}\Big|_{s=0} \langle \ee_{n+1}, N\rangle = \langle
\bar\nabla_{\partial_s}\ee_{n+1}, N\rangle
+\langle \ee_{n+1}, \bar\nabla_{\partial_s} N\rangle=\langle \ee_{n+1}, -\nabla u- \mathcal{W} T\rangle.
\end{equation}
Since
\begin{equation}
\label{der-H}
\frac{{\rm d}}{{\rm d}s}\Big|_{s=0} H = \DELTA u + |A|^2 \, u + \mathcal{L}_T H
\end{equation}
and $\nabla \eta = \ee_{n+1}^\top$, we obtain for normal variations (when $T=0$)
\begin{align*}
 \frac{{\rm d}}{{\rm d}s}\big(H-\langle \ee_{n+1}, N\rangle\big) 
 &=\DELTA u -\langle \ee_{n+1},  \nabla u\rangle +|A|^2 \, u \\
 &=\DELTA^{f} u + |A|^2 u.
\end{align*}
This finishes the proof of the proposition. \hfill $\square$

\vspace{3mm}

The previous result has important consequences. 
It means that  a {hyper}surface $M\subset\RR^{n+1}$ 
is a translator if and only if
it is minimal with respect to the Riemannian metric
\[
    g_{ij}(x_1, \ldots,x_{n+1}) = \exp\left(-\frac2n \, x_{n+1} \right) \delta_{ij}.
\]
Although the metric $g_{ij}$ is not complete
(notice that the length of vertical half-lines in the direction of  $\ee_{n+1}$ is finite), we can apply
all the local results of the theory of minimal hypersurfaces in Riemannian manifolds.
Thus we can freely use curvature estimates and compactness theorems from minimal surface theory;
cf.~\cite[Chapter 3]{white-intro}.
In particular, if $M$ is a graphical translator, then (since vertical translates of it are also $g$-minimal)
$\left<\ee_3,\nu\right>$ is a nowhere vanishing Jacobi field, so $M$ is a stable $g$-minimal surface.
It follows that any sequence $M_i$ of complete translating graphs in $\RR^3$ has a subsequence that
converges smoothly to a translator $M$.  Also, if a translator $M$ is the graph of
a function $u:\Omega\to\RR$, then $M$ and its vertical translates from a $g$-minimal
foliation of $\Omega\times\RR$, from which it follows that $M$ is $g$-area minimizing in $\Omega\times\RR$,
and thus that if $K\subset \Omega\times \RR$ is compact, then the $g$-area of $M\cap K$ is
at most $1/2$ of the $g$-area of $\partial K$.  Hence, if we consider sequences
of translators that are manifolds-with-boundary, then the area bounds described above
together with standard compactness theorems for minimal surfaces (such as those
in~\cite{white-curvature-estimates, white-controlling}) give smooth, subsequential convergence,
including at the boundary.  This has been a crucial tool in \cite{mpss, himw, families-1, families-2}
(The local area bounds and bounded topology mean that the only boundary singularities
that could arise would be boundary branch points.   In the situations that occur in these papers, 
obvious barriers preclude boundary branch points.)

The situation for higher dimensional translating graphs is more subtle; (see \cite[Appendix A]{himw} and \cite{nino}).

\section{Examples of translators}

Besides the grim reapers that we have already described, the last decades have witnessed the appearance
of numerous examples of translators. 
Clutterbuck, Schn\"urer and Schulze \cite{CSS} (see also \cite{Altschuler-Wu}) proved that there exists an entire graphical translator in $\R^{n+1}$ which is rotationally symmetric, strictly convex  with translating velocity $-\ee_{n+1}$. This example is known as the {\bf translating paraboloid} or {\bf bowl soliton}. Moreover, they classified all the translating solitons of revolution, giving a one-parameter family $\{W^n_{\lambda}\}_{\lambda>0}$ of rotationally invariant cylinders called {\bf translating catenoids}. The parameter $\lambda$ 
{controls}
 the size of the neck of each translating soliton. The limit, as $\lambda \to 0$, of $W^n_{\lambda}$ consists of two
 superimposed copies of the bowl soliton with a singular point at the axis of symmetry. Furthermore, all these hypersurfaces have the following asymptotic expansion as r approaches infinity:
\begin{equation*}
\frac{r^2}{2(n-1)}-\log r+O(r^{-1}),
\end{equation*}
where $r$ is the distance function in $\mathbf{R}^{n}$. 
\begin{figure}[htbp]
\begin{center}
\includegraphics[width=.35\textwidth]{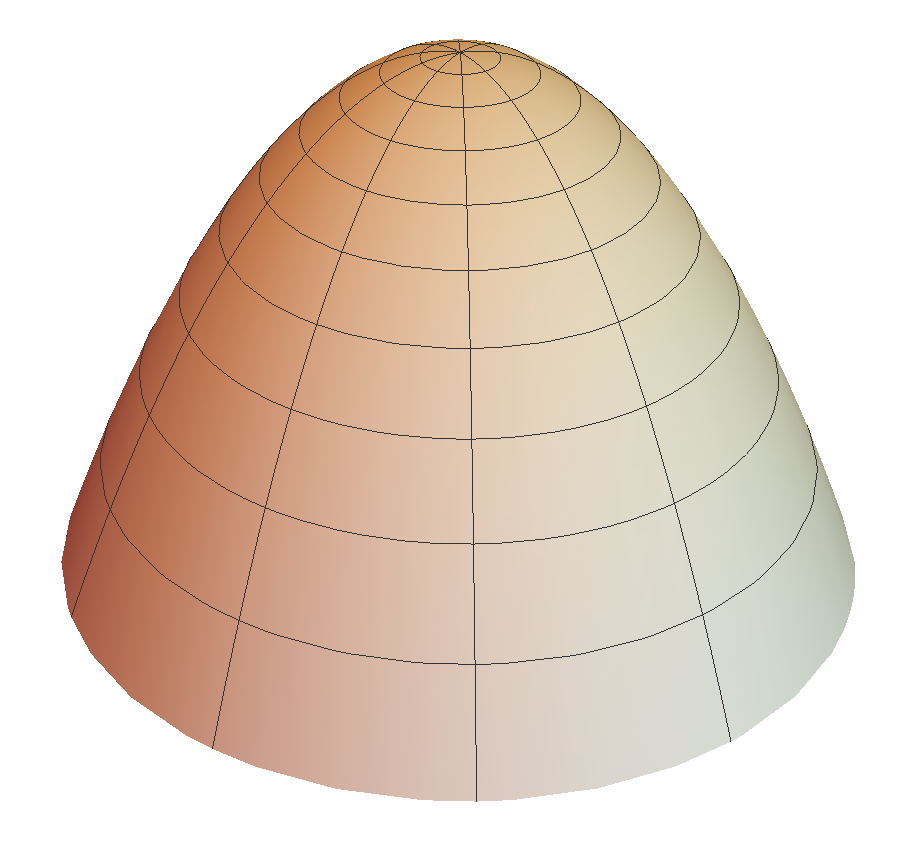} \hfill\includegraphics[width=.35\textwidth]{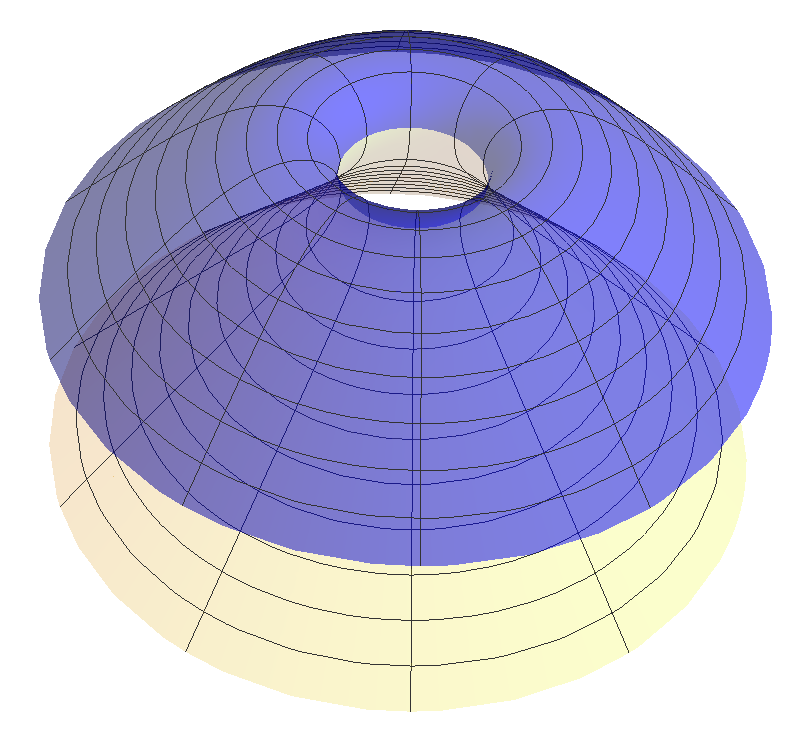}
\caption{The bowl soliton in $\RR^3$ and the translating catenoid for $\lambda=2.$}
\label{catenoid}
\end{center}
\end{figure}
These rotationally symmetric translating catenoids can be seen as the desingularization of two paraboloids connected by a small neck of some radius. 

{
Recall that the Costa-Hoffman-Meeks surfaces can be regarded as desingularizations
of a plane and catenoid: a sequence of Costa-Hoffman-Meeks surfaces with genus tending to infinity
converges (if suitably scaled) to the union of a catenoid and the plane through its waist.
This suggests that one try to construct translators by desingularizing the union
of a translating catenoid and a bowl solition.
D\'avila, del Pino, and Nguyen~ \cite{DDPN} were able to do that (for large genus) by
glueing methods, replacing the circle  of intersection by a 
surface similar to the singly periodic Scherk minimal surface (see also \cite{smith}.)
Previously, Nguyen  in \cite{Nguyen09},\cite{Nguyen13} and \cite{Nguyen15} had used similar techniques to desingularize the intersection of a grim reaper and a plane. In this way she obtained a complete periodic embedded translator of infinite genus, that 
she called a translating trident.
} 

Once  this abundance of translating solitons is guaranteed, there  arises the need to classify them. 
One of the first classification results was given by X.-J. Wang in \cite{wang}. He  characterized the bowl soliton as the only convex translating soliton which is an entire graph. 

Very recently, J. Spruck and L. Xiao \cite{spruck-xiao} have proved that a complete translating soliton which is graph over a domain in $\R^2$ must be convex (see Section~\ref{spruck-xiao-section} below.) So, combining both results we have:
{
\begin{theorem}
The bowl soliton is 
the only translator that is an entire graph over $\RR^2$.
\end{theorem}
} 

Using the Alexandrov method of moving hyperplanes, Martín, Savas-Halilaj, and Smoczyk \cite{mss} showed that the bowl soliton is the only translator (not assumed to be graphical) 
that has one end and is $C^\infty$-asymptotic to a bowl soliton. 
Hershkovits \cite{Her18} improved this by showing uniqueness of the bowl soliton among (not necessarily graphical) 
translators that have one cylindrical end (and no other ends).   Haslhofer \cite{Has15} proved a related result
in higher dimensions: he showed that any translator in $\R^{n+1}$ that is noncollapsed and uniformly $2$-convex must
be the $n$-dimensional bowl soliton. At this point, we would like to mention the recent classification result of Brendle and Choi \cite{brendle}. They
prove that the rotationally symmetric bowl soliton is the only noncompact ancient solution of mean curvature flow in $\R^3$ which is strictly convex and noncollapsed.

Martín, Savas-Halilaj and Smoczyk also obtained one of the first characterizations of the family of tilted grim reapers:
\begin{theorem} \label{th:AH}\cite{mss}
Let $M$ be a connected translating soliton in $\R^{n+1}$, $n\geq 2$, such that the function $|A|^2H^{-2}$ has a local maximum in $\{x\in M: H(x)\ne 0\}$. 
Then $M$ is a tilted
grim reaper.
\end{theorem}



\section{Graphical translators}

If a translator $M$ is the graph of function $u:\Omega\subset\RR^n\to\RR$, we will
say that $M$ is a {\bf translating graph}; in that case, we also refer to the
function $u$ as a translator, and we say that $u$ is complete if its graph
is a complete submanifold of $\RR^{n+1}$.  
Thus $u:\Omega\subset\RR^n\to\RR$ is a translator if and only if
it solves the 
translator equation (the nonparametric form of~\eqref{general-translator-equation}):
\begin{equation}\label{divergence-translator-equation}
  D_i\left( \frac{D_iu}{\sqrt{1+ |Du|^2}} \right) = -\frac1{\sqrt{1+|Du|^2}}.
\end{equation}
The equation can also be written as
\begin{equation}\label{translator-equation}
   (1 + |Du|^2) \DELTA u - D_iu\,D_ju\,D_{ij}u + |Du|^2 +1 = 0.
\end{equation}

In a recent preprint, we classify all complete translating graphs in $\RR^3$.
In two other papers~\cite{families-1, families-2}, 
we construct new families of complete, properly embedded (non-graphical) translators:
a two-parameter family of translating annuli,
examples that resemble Scherk's minimal surfaces, and examples that resemble helicoids. In \cite{families-2}, we also construct several new families of complete translators that are obtained as limits of the Scherk-type
translators mentioned above.
They include a $1$-parameter family of single periodic surfaces
called Scherkenoids (see Fig.~\ref{scherk-1-2}) and a simply-connected translator called 
the pitchfork translator (see Fig.~\ref{scherk-2-2}). The pitchfork translator
resembles Nguyen's translating tridents \cite{Nguyen09} (see also \cite{families-1.5}):
like the tridents, it is asymptotic to a plane as $z\to\infty$
and to three parallel planes as $z\to -\infty$.
 However, the pitchfork
has genus $0$, whereas the tridents have infinite genus.
\begin{figure}[htbp]
\begin{center}
\includegraphics[width=.45\textwidth]{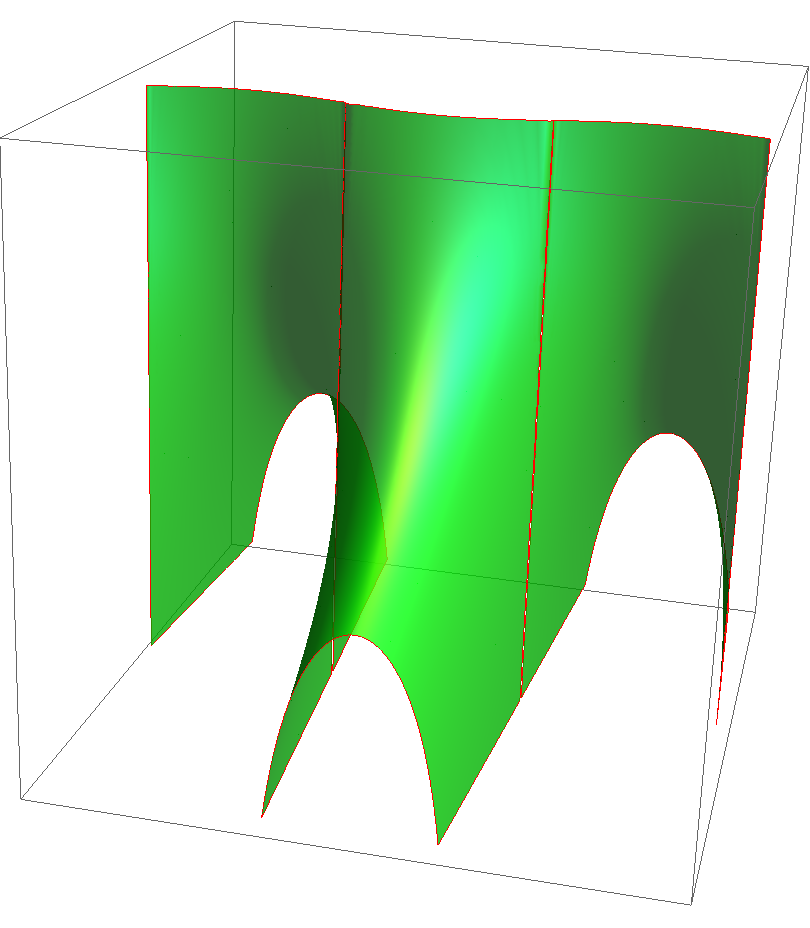}
\caption{A Scherkenoid is a singly periodic translator.
  As $z\to\infty$, it is asymptotic to a plane, and as $z\to-\infty$, it is asymptotic
  to an infinite family of parallel planes.
  There is a one-parameter family of such Scherkenoids, the parameter 
  being the angle between the upper plane and the lower planes.}
\label{scherk-1-2}
\end{center}
\end{figure}
\begin{figure}[htbp]
\begin{center}
\includegraphics[width=.45\textwidth]{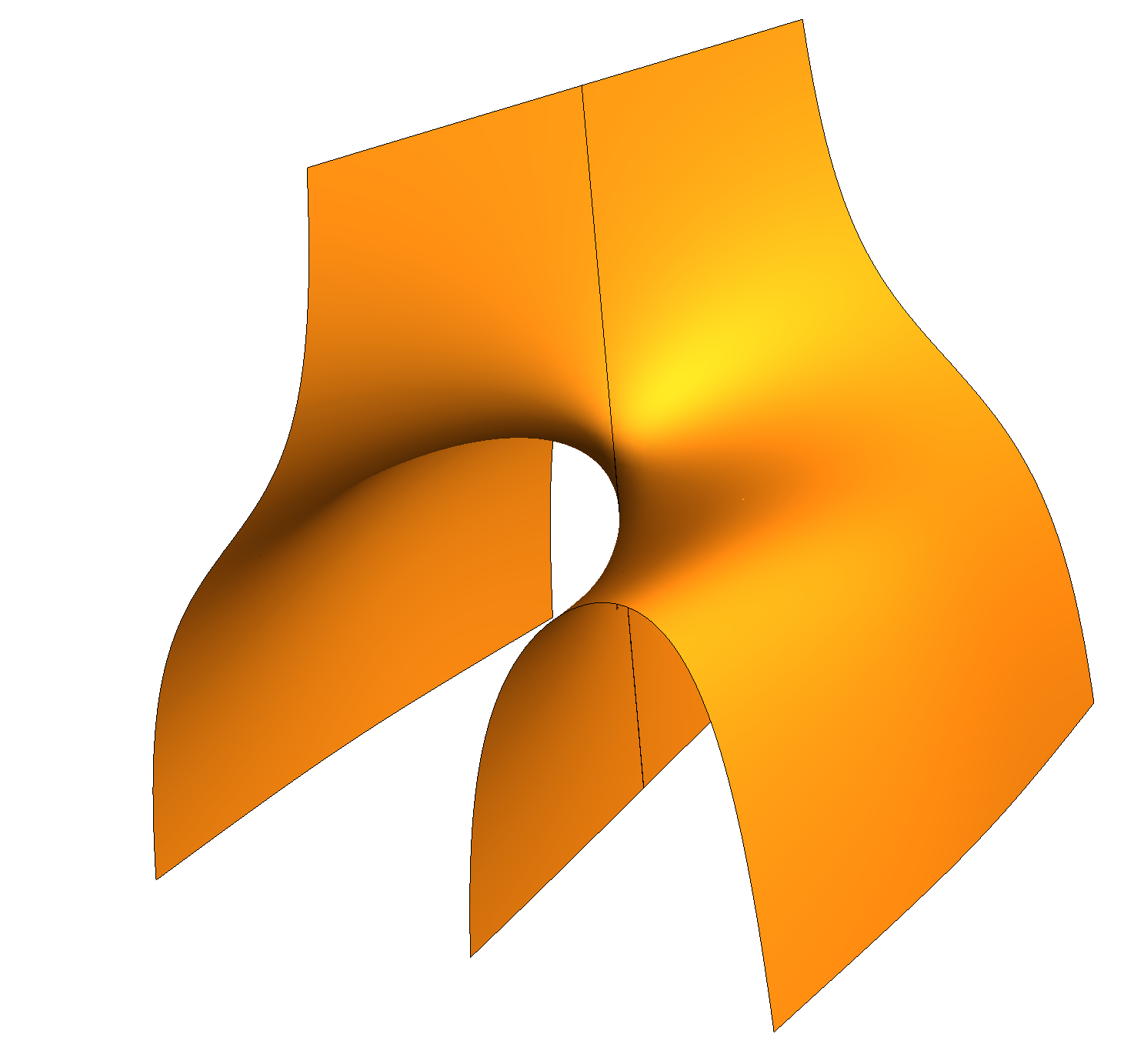}
\caption{The pitchfork translator
is a simply connected translator that is asympotic to a plane as $z\to\infty$
and to three parallel planes as $z\to - \infty$.}
\label{scherk-2-2}
\end{center}
\end{figure}

As a consequence of Theorem~\ref{th:AH}, we have that every translator $\RR^3$ with zero Gauss curvature is
 a grim reaper surface, a tilted grim reaper surface, or a vertical plane.

In addition to the examples described in the previous section, Ilmanen (in unpublished work) proved that for
each $0< k < 1/2$, there is a translator 
$
   u: \Omega\to \RR
$
with the following properties:
 $u(x,y)\equiv u(-x,y)\equiv u(x,-y)$, 
$u$ attains its maximum at $(0,0)\in \Omega$, and
\[
   D^2u(0,0)  
   =\begin{bmatrix} -k &0 \\ 0 &-(1-k) \end{bmatrix}.
\]
The domain $\Omega$ is either a strip $\RR\times(-b,b)$ or $\RR^2$.
He referred to these examples as {\bf $\DELTA$-wings}.
As $k\to 0$, he showed that the examples converge to the grim reaper surface.
Uniqueness (for a given $k$) was not known.   It was also not known
which strips $\RR\times (-b,b)$ occur as domains of such examples.
 The main result in \cite{himw} is the following:

\begin{theorem}\label{main-theorem}
For every $b>\pi/2$, there is (up to translation) a unique complete, strictly convex
translator 
$
   u^b: \RR\times (-b,b)\to\RR.
$
Up to isometries of $\RR^2$, the only other complete translating graphs in $\RR^3$
are the grim reaper surface, the tilted grim reaper surfaces, and the bowl soliton.
\end{theorem}

{Although the paper~\cite{himw} is primarily about translators in $\RR^3$, the last sections 
extend Ilmanen's original proof to get $\DELTA$-wings in $\RR^{n+1}$ that have 
  prescribed principal curvatures at the origin.
  } 
  For $n\ge 3$, the examples include entire graphs that are not rotationally invariant.
At the end of the paper,  we modify the construction to produce a family of $\DELTA$-wings in $\RR^{n+2}$
 over any given slab of width $>\pi$.  
 {See~\cite{wang} for a different construction
 of some higher dimensional graphical translators. 
 } 
\begin{figure}[htbp]
\begin{center}
\includegraphics[width=.45\textwidth]{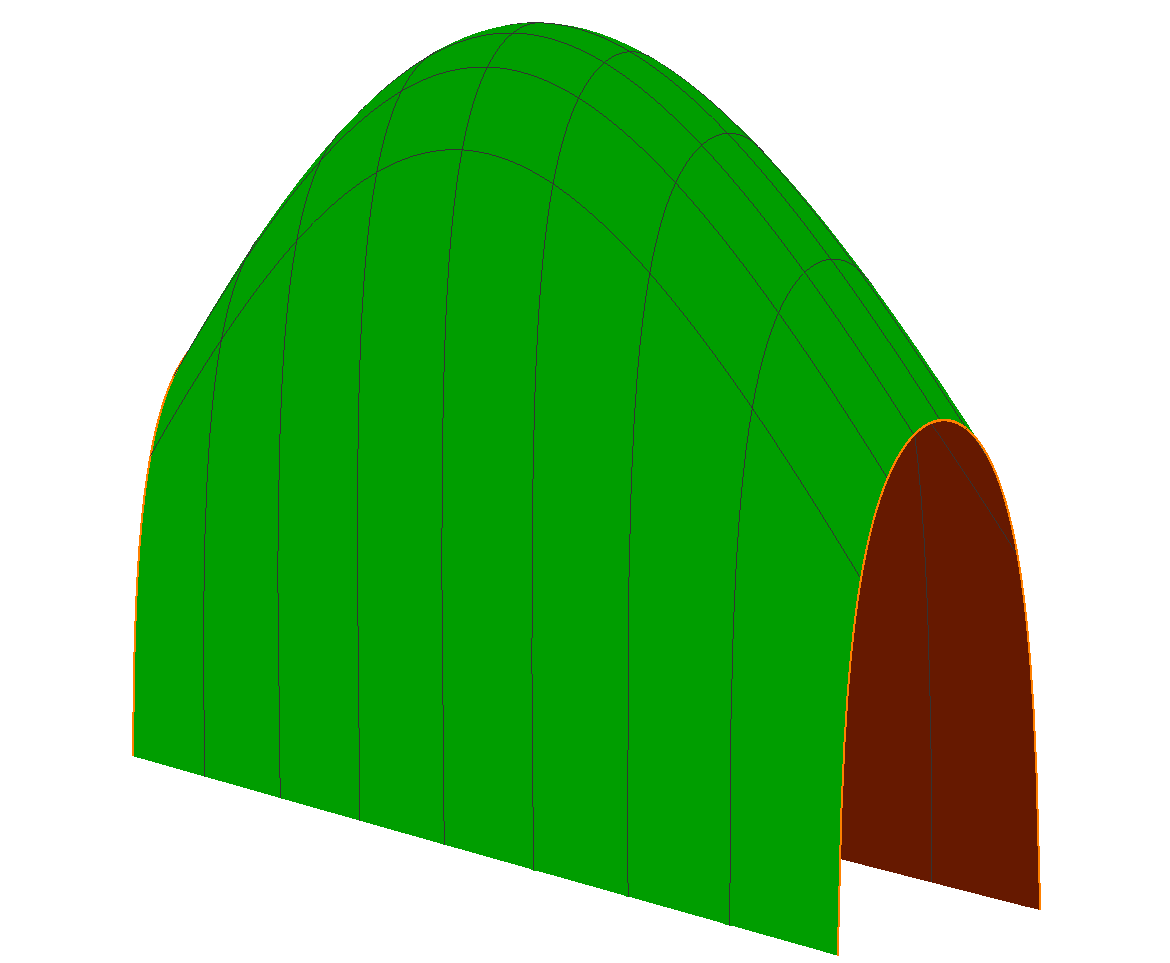}
\caption{The $\DELTA$-wing of width $\sqrt{2} \, \pi.$ As $y \to\pm \infty,$ this $\DELTA$-wing
is asymptotic to the tilted grim reapers $\grim_{-\frac{\pi}{4}}$ and $\grim_{\frac{\pi}{4}}$,
respectively.}
\label{intro}
\end{center}
\end{figure}

\section{The Spruck-Xiao Convexity Theorem}\label{spruck-xiao-section}
One of the fundamental results in the recent development of soliton theory has been the paper by Spruck 
and Xiao \cite{spruck-xiao}, where they proved that complete graphical translators 
{ (or, more generally, complete translators of positive mean curvature)}
are convex. The ideas contained 
in this paper are really inspiring and  we would like to provide a slightly simplified exposition of their proof.

At any non-umbilic point, we let $\kappa_1>\kappa_2$ be the principal curvatures
and $H=\kappa_1+\kappa_2>0$ be the mean curvature.
We let $v_1$ and $v_2$ be the principal direction unit vector fields, so
\[
  \text{$\kappa_i \equiv A(v_i,v_i)$ and $A(v_1,v_2)\equiv 0$}.
\]
Note $\nabla_uv_1$ is perpendicular to $v_1$.  Thus
\begin{equation}\label{derivatives-of-v1}
   \nabla_{v_1}v_1 = \alpha_1 v_2, \qquad \nabla_{v_2}v_1=\alpha_2v_2
\end{equation}
for some functions $\alpha_1$ and $\alpha_2$.
Since $0=\nabla_{v_i}(v_1\cdot v_2) = (\nabla_{v_i}v_1)\cdot v_2 + v_1\cdot(\nabla_{v_i}v_2)$,
we see that
\begin{equation}\label{derivatives-of-v2}
  \nabla_{v_1}v_2 = -\alpha_1v_1, \qquad \nabla_{v_2}v_2 = -\alpha_2 v_1.
\end{equation}
Thus
\begin{align*}
\nabla_{u} \kappa_1
&= \nabla_{u}A(v_1,v_1) \\
&= (\nabla_{u}A)(v_1,v_1) + 2A(\nabla_{u}v_1,v_1).
\end{align*}
But $\nabla_uv_1$ is perpendicular to $v_1$, so $A(\nabla_{u}v_1,v_1)\equiv 0$.
Thus
\begin{equation}\label{grad-k}
   \nabla_{u}\kappa_1 = (\nabla_{u}A)(v_1,v_1).
\end{equation}
In particular,
\begin{equation}\label{key}
  \nabla_i \kappa_j = h_{jj,i},
\end{equation}
{ where $h_{ij}=A(v_i,v_j)$ and $h_{ij,k}= (\nabla_{v_k}A)(v_i,v_j)$.}
From $A(v_1,v_2)\equiv 0$, we see that
\begin{align*}
h_{12,i}
&= (\nabla_i)A(v_1,v_2) \\
&= \nabla_i (A(v_1,v_2)) - A(\nabla_iv_1,v_2) - A(v_1,\nabla_iv_2) \\
&= 0 -\alpha_i h_{22} + \alpha_i h_{11} \\
&= (\kappa_1-\kappa_2)\alpha_i
\end{align*}
{by~\eqref{derivatives-of-v1} and~\eqref{derivatives-of-v2}.}
Thus
\begin{equation}\label{alpha-i}
\alpha_i
=
\frac{h_{12,i}}{\kappa_1- \kappa_2}.
\end{equation}

Also, if we let $u_i=v_i$ at a particular point and extend by parallel transport on radial geodesics, we have
\begin{align*}
\DELTA \kappa_1 
&= \nabla_{u_i}\nabla_{u_i} (A(v_1,v_1) )\\
&= \nabla_{u_i} ( (\nabla_{u_i}A)(v_1,v_1)) \\
&= (\DELTA A)(v_1,v_1) + 2 (\nabla_{u_i}A) (\nabla_{u_i}v_1,v_1).
\end{align*}
Thus
\begin{align*}
\DELTA\kappa_1 
&= (\DELTA A)(v_1,v_1) + 2(\nabla_iA)(\nabla_iv_1,v_1) \\
&= (\DELTA A)(v_1,v_1) + 2(\nabla_iA) (\alpha_i v_2,v_1) \\
&= (\DELTA A)(v_1,v_1) + 2\alpha_i h_{12,i} \\
&= (\DELTA A)(v_1,v_1) + 2\frac{ (h_{12,1})^2 + (h_{12,2})^2}{\kappa_1- \kappa_2} \\
&= (\DELTA A)(v_1,v_1) + \frac{2Q^2}{\kappa_1 - \kappa_2},
\end{align*}
where
 \begin{equation}\label{Q-squared}
   Q^2 := (h_{12,1})^2 + (h_{12,2})^2=(h_{11,2})^2 + (h_{22,1})^2.
 \end{equation}
(The second equality follows from the Codazzi equations.) 

Now suppose the surface is a translator.  Then, we have that (see \cite[Lemma 2.1]{spruck-xiao} or \cite[Lemma~2.1]{mss}):
$$ \DELTA A- \nabla_{e_3^T} A+|A|^2 \, A=0.$$
Hence,
\begin{align*}
 (\DELTA A)(v_1,v_1)
 &=
 -|A|^2A(v_1,v_1) + (\nabla_{e_3^T} A)(v_1,v_1) \\
 &= 
 -|A|^2 \kappa_1 + \nabla_{e_3^T}\kappa_1
 \end{align*}
 by~\eqref{grad-k}.
 Thus
 \begin{equation}\label{drift-of-kappa1}
 \DELTA^f \kappa_1 
 = -|A|^2\kappa_1 + \frac{2Q^2}{\kappa_1-\kappa_2}.
 \end{equation}
 Recall that the {\it drift Laplacian} (see \eqref{deltaT-op}) is given by 
 $$\DELTA^f:= \DELTA- \langle \nabla f, \nabla \cdot \rangle= \DELTA -\langle e_3 , \nabla \cdot \rangle.$$
 Likewise,
 \[
 \DELTA^f \kappa_2 = -|A|^2\kappa_2 - \frac{ 2Q^2}{\kappa_1-\kappa_2}.
 \]
 
 Adding these gives
 \begin{equation}\label{drift-of-H}
   \DELTA^f H = - |A|^2 H.
 \end{equation}
  
 From~\eqref{drift-of-kappa1} and~\eqref{drift-of-H},
 \[
   \kappa_1\DELTA^fH - H \DELTA^f\kappa_1 = \frac{-2HQ^2}{\kappa_1- \kappa_2}.
 \]
 
 Thus
 {
 \begin{equation}\label{drift-of-H-over-k}
 \begin{aligned}
 \DELTA^f\left( \frac{H}{\kappa_1} \right)
 &=
 \frac{ \kappa_1 \DELTA^f H - H \DELTA^f\kappa_1}{\kappa_1{}^2} 
 -
 2\,\frac{\nabla \kappa_1}{\kappa_1}\cdot\nabla\left(\frac{H}{\kappa_1}\right)   \\
 &=
  \frac{-2HQ^2}{(\kappa_1-\kappa_2)(\kappa_1)^2}
 -
 2\,\frac{\nabla \kappa_1}{\kappa_1}
 \cdot 
 \nabla\left(\frac{H}{\kappa_1} \right),
 \end{aligned}
 \end{equation}
so
\begin{equation}\label{drift-of-H-over-k-inequality}
\DELTA^f\left( \frac{H}{\kappa_1} \right)
+
2\,\frac{\nabla \kappa_1}{\kappa_1} \cdot
 \nabla\left(\frac{H}{\kappa_1} \right) 
\le 
0.
 \end{equation}
}  

\begin{theorem}[Spruck-Xiao]
Let $M\subset \RR^3$ be a complete translator with $H>0$.  Then $M$ is convex.
\end{theorem}

\begin{proof}
Suppose the theorem is false.
Then
\begin{equation}\label{hypothesis}
  \eta := \inf \frac{H}{\kappa_1} = \inf \left(1 + \frac{\kappa_2}{\kappa_1}\right) \in [0,1).
\end{equation}
{
Note that the set of points where $H/\kappa_1<2$ contains no umbilic points,
which implies that $H/\kappa_1$ is smooth on that set and that we can apply the formulas 
in this section (\S6), which were derived assuming that $\kappa_1>\kappa_2$.
} 
\newline

\noindent
{\bf Step 1}: The infimum $\eta$ is not attained.
For suppose it is attained at some point.  
{
Then by~\eqref{drift-of-H-over-k-inequality} and the strong
minimum principle (see~\cite[Theorem~3.5]{gilbarg-trudinger}
or~\cite[Chapter~6, Theorem~3]{evans-book}), $H/\kappa_1$ is constant on $M$. 
} 
Therefore $\kappa_2/\kappa_1= H/\kappa_1-1$ 
is constant on $M$.
Since $H>0$ and since $H/\kappa_1$ is constant, 
we see { from~\eqref{drift-of-H-over-k}} that $Q\equiv 0$, i.e., (see~\eqref{Q-squared})
 $h_{12,1}\equiv h_{12,2}\equiv 0$.
Hence by~\eqref{alpha-i}, \eqref{derivatives-of-v1} and~\eqref{derivatives-of-v2}, the
frame $\{v_1,v_2\}$ is parallel, so $\kappa_1\kappa_2\equiv 0$,
{
contadicting~\eqref{hypothesis}.  Thus the infimum is not attained.
} 
\newline

\noindent {\bf Step 2}: If $p_n\in M$ is a sequence with $H(p_n)/\kappa_1(p_n)\to \eta$, 
then (after passing to a subsequence)
$M - p_n$ converges smoothly to a limit by the curvature estimates mentioned at the end of Section \ref{sec:translators}. We claim that the limit must be a vertical plane.
For suppose not.  Then (after passing to a subsequence) $M-p_n$ converges smoothly to 
a complete translator $M'$ with $H>0$.  
By the smooth convergence,  $H/\kappa_1$ attains its minimum value on $M'$ at the origin, and that
minimum value is $\eta$, contradicting Step 1.
\newline
\newline
\noindent
{\bf Step 3}:
Now we apply the Omori-Yau maximum principle (see Theorem~\ref{Omori-Yau} below) to get a sequence $p_n\in M$ such that
\begin{align}
\frac{H}{\kappa_1} = 1 + \frac{\kappa_1}{\kappa_2} &\to \eta,                                      \label{omori-function}\\
\nabla \left(\frac{H}{\kappa_1} \right) &\to 0,                              \label{omori-gradient} \\
\DELTA \left( \frac{H}{\kappa_1} \right) &\to  \delta \in [0,\infty].    \label{omori-laplacian}
\end{align}
From~\eqref{omori-gradient} and~\eqref{omori-laplacian}, we see that
\begin{equation}\label{omori-drift}
\DELTA^f \left( \frac{H}{\kappa_1} \right) \to  \delta \in [0,\infty].
\end{equation}
By Step 2, we can assume that $M-p_n$ converges smoothly to a vertical plane.
For the rest of the proof, any statement that some quantity tends to a limit
refers only to the quantity at the points $p_n$.

Since $A$ is a quadratic form with eigenvalues $\kappa_1$ and $\kappa_2$,
$A/\kappa_1$ is a a quadratic form with eigenvalues $1$ and $\kappa_2/\kappa_1= H/\kappa_1-1$.
Thus (by passing to a subsequence) we can assume that $A/\kappa_1$ (at $p_n$) converges
to a quadratic form with eigenvalues $1$ and $\eta-1$. (Note that the eigenvalue $\eta-1$ is negative
by Hypothesis~\eqref{hypothesis}.)

Recall that
\[
   \nabla H = A(\ee_3^T,\cdot). 
\]
{(See, for example, \cite[Lemma 2.1]{mss}.)}
Since $\ee_3^T\to\ee_3$, we see that $\nabla H/\kappa_1$ (at $p_n$) converges
 to a nonzero vector $N$:
\begin{equation}\label{N}
  \frac{\nabla H}{\kappa_1}\to N\ne 0.
\end{equation}
Now 
\begin{equation}\label{gradient}
 \nabla\left(\frac{H}{\kappa_1}\right)
 =
 \frac{\nabla H}{\kappa_1} - \frac{H}{\kappa_1}\frac{\nabla \kappa_1}{\kappa_1}.
 \end{equation}
By Omori-Yau (see~\eqref{omori-gradient}), this tends to $0$, so 
\begin{equation}\label{pork-one}
  \frac{H}{\kappa_1}\frac{\nabla \kappa_1}{\kappa_1} \to N,
\end{equation}
or, equivalently (see~\eqref{key}),
\begin{equation}\label{re-pork-one}
 \frac{H}{\kappa_1} \frac{h_{11,i}}{\kappa_1} \to N_i \qquad (i=1,2),
\end{equation}
where $N_i = N\cdot v_i$.

We can rewrite~\eqref{gradient} as
\begin{align*}
 \nabla\left( \frac{H}{\kappa_1} \right)
 &=
 \frac{\nabla\kappa_1 + \nabla \kappa_2}{\kappa_1} - \frac{H}{\kappa_1}\frac{\nabla\kappa_1}{\kappa_1}
 \\
 &=
 \left(1 -\frac{H}{\kappa_1} \right) \left( \frac{\nabla\kappa_1}{\kappa_1} \right) + \frac{\nabla\kappa_2}{\kappa_1}.
 \end{align*}
 Multiply by $H/\kappa_1$:
 \[
 \frac{H}{\kappa_1} \nabla\left( \frac{H}{\kappa_1} \right) 
 =
 \left(1 -\frac{H}{\kappa_1} \right) \left(\frac{H}{\kappa_1} \frac{\nabla\kappa_1}{\kappa_1} \right) 
 + \frac{H}{\kappa_1} \frac{\nabla\kappa_2}{\kappa_1}
\]
By Omori-Yau (see~\eqref{omori-gradient}), this tends to $0$, so (using~\eqref{pork-one}),
\begin{equation}\label{pork-two}
 \frac{H}{\kappa_1} \frac{\nabla\kappa_2}{\kappa_1} \to (\eta-1)N,
\end{equation}
or, equivalently {(by~\eqref{key})},
\begin{equation}\label{re-pork-two}
   \frac{H}{\kappa_1} \frac{h_{22,i}}{\kappa_1} \to (\eta - 1) N_i \qquad (i=1,2).
\end{equation}
Combining~\eqref{re-pork-one} with $i=2$ and~\eqref{re-pork-two} with $i=1$ gives
\begin{equation}\label{lambda}
  \left(\frac{H}{\kappa_1}\right)^2
     \left(  \frac{Q}{\kappa_1}  \right)^2
 \to
 (N_2)^2 + (\eta-1)^2 (N_1)^2 := \lambda^2 > 0.
 \end{equation}
 Note that $\lambda^2>0$ because $N\ne 0$ and $\eta<1$ by Hypothesis~\eqref{hypothesis}.
 
 Now multiply~\eqref{drift-of-H-over-k} by $H/\kappa_1$:
\begin{equation*}
\left(\frac{H}{\kappa_1}\right) \DELTA^f\left( \frac{H}{\kappa_1} \right)
 =
   \frac{-2}{1 - \frac{\kappa_2}{\kappa_1}} \left( \frac{H}{\kappa_1} \right)^2 \left( \frac{Q}{\kappa_1}\right)^2
 -
2 \left( \frac{H}{\kappa_1} \frac{\nabla \kappa_1}{\kappa_1}\right)\cdot  \nabla\left( \frac{H}{\kappa_1} \right). 
\end{equation*}
Using~\eqref{omori-function} and~\eqref{omori-drift} for the left side,
\eqref{omori-function} and~\eqref{lambda} for the first term on the right,
and~\eqref{pork-one} and~\eqref{omori-gradient} for the second term, we can let $n\to\infty$
to get:
\[
 \eta \, \delta  \le  \frac{-2}{2-\eta}\lambda^2 + 0,
\]
a contradiction (since $\eta$ and $\delta$ are nonnegative).
\end{proof}

We used the Omori-Yau Theorem (see, for example, \cite{Alias-et-al}):

\begin{theorem}[Omori-Yau Theorem]\label{Omori-Yau}
Let $M$ be a complete Riemannian manifold with Ricci curvature bounded below.
Let $f:M\to \RR$ be a smooth function that is bounded below.  Then there is a sequence $p_n$ in $M$
such that
\begin{gather}
f(p_n) \to \inf_Mf, \\
\nabla f(p_n) \to 0, \\
\liminf \DELTA f(p_n) \ge 0.
\end{gather}
The theorem remains true if we replace the assumption that $f$ is smooth
by the assumption that $f$ is smooth on $\{ f< \alpha \}$ for some $\alpha>\inf_Mf$.
\end{theorem}

To see the last assertion, let $\phi:\RR\to \RR$ be a smooth, monotonic function
such that $\phi(t)=0$ for $t\ge \alpha$ and such that $\phi\equiv 1$ on an open interval containing $\inf_Mf$.
Then $\phi\circ f$ is smooth, so the Omori-Yau Theorem holds for $\phi\circ f$, from which it follows immediately that
the Omori-Yau Theorem also holds for $f$.

In our case, the function $H/\kappa_1$ is smooth except at umbilic points. At such points,
$H/\kappa_1=2$.   Since we assumed that the infimum was $<1$, we could invoke
the Omori-Yau Theorem.


\section{Characterization of translating graphs in $\RR^3$}

As we mentioned before, the authors of these notes have obtained the complete classification 
of the complete graphical translators in Euclidean $3$-space. 

Recall that by translator we mean a smooth function $u: \Omega \rightarrow \R$ such that 
$M=\operatorname{Graph}(u)$
is a translator.  Then $u$ must be  solution of the equation:
\begin{equation} \label{eq:trans}
(1+u_y^2) \, u_{xx}-2 u_x \, u_y \, u_{xy}+(1+u_x^2)  \, u_{yy}+u_x^2+u_y^2+1=0.
\end{equation} 
If we impose that $M$ is complete, then we will say that $u$ is a {\bf complete translator}.
In this setting Shahriyari~\cite{shari} proved in her thesis the following
\begin{theorem}[Shahriyari]
If $M$ is complete, then $\Omega$ must be a strip, a halfspace, or all of $\R^2$.
\end{theorem} 
In~\cite{wang}, X. J. Wang proved that the only entire convex translating graph  is the
bowl soliton, and that there are no complete translating graphs defined over
halfplanes.  Thus by the Spruck-Xiao Convexity Theorem, 
the bowl soliton is the only complete translating graph defined over a plane or halfplane.

Hence, it remained to classify the translators $u:\Omega\to \RR$ whose domains
are strips.
Our main new contributions in this line are:
\begin{enumerate}
\item\label{existence-item} For each $b>\pi/2$, we prove (\cite{himw}[Theorem~5.7])
existence and uniqueness (up to translation)
of a complete translator $u^b: \RR \times (-b,b) \to \RR$ that is not a tilted grim reaper.
We call $u^b$ the {\bf $\DELTA$-wing} of width $2b$.
\item We give a simpler proof (see \cite{himw}[Theorem~6.7]) 
that there are no complete graphical translators in $\R^3$ defined over halfplanes in $\R^2$.
\end{enumerate}
Furthermore, there are no complete translating graphs defined over strips of width $<\pi$ (see \cite{spruck-xiao, bourni-et-al}),
and the grim reaper surface is the only translating graph over a strip of width $\pi$ (see \cite{himw}).
Consequently, we have a classification: every complete, translating graph in $\RR^3$ is
one of the following: a grim reaper surface or tilted grim reaper surface, a $\DELTA$-wing,
or the bowl soliton.

We remark that Bourni, Langford, and Tinaglia gave a different proof of the existence
(but not uniqueness) of the $\DELTA$-wings in~\eqref{existence-item} \cite{bourni-et-al}.



\end{document}